\documentclass[11pt,a4paper]{article}

\usepackage{inputenc}
\usepackage{amsmath}
\usepackage{bm}
\usepackage{bbold}
\usepackage{amsthm}

\usepackage{hyperref}

\title{On Evaluation of the Mean Service Cycle Time\\ in Tandem Queueing Systems\thanks{Applied Statistical Science V / Ed. by M. Ahsanullah, J. Kennyon, and S. K. Sarkar,
Nova Science Publishers, Huntington, NY, 2001, pp. 145-155.}}

\author{N.~K.~Krivulin\thanks{Faculty of Mathematics and Mechanics, St.~Petersburg State University, 28 Universitetsky Ave., St.~Petersburg, 198504, Russia, Nikolai.Krivulin@pobox.spbu.ru.}  \thanks{The work was partially supported by the Russian Foundation for Basic Research, Grant~\#00-01-00760.}
\and
V.~B.~Nevzorov\thanks{Faculty of Mathematics and Mechanics, St.~Petersburg State University, 28 Universitetsky Ave., St.~Petersburg, 198504, Russia, Valery.Nevzorov@pobox.spbu.ru.} \thanks{The work was partially supported by the Russian Foundation for Basic Research, Grant \#99-01-00732}
}

\date{}

\newtheorem{theorem}{Theorem}
\newtheorem{lemma}[theorem]{Lemma}
\newtheorem{corollary}[theorem]{Corollary}

\begin{document}

\maketitle

\begin{abstract}
The problem of exact evaluation of the mean service cycle time in tandem
systems of single-server queues with both infinite and finite buffers is
considered. It is assumed that the interarrival and service times of customers
form sequences of independent and identically distributed random variables
with known mean values. We start with tandem queues with infinite buffers, and
show that under the above assumptions, the mean cycle time exists.
Furthermore, if the random variables which represent interarrival and service
times have finite variance, the mean cycle time can be calculated as the
maximum out from the mean values of these variables. Finally, obtained results
are extended to evaluation of the mean cycle time in particular tandem systems
with finite buffers and blocking.
\\

\end{abstract}

\section{Introduction}

We consider tandem systems of single-server queues with both infinite and
finite buffers. The interarrival and service times of customers are assumed to
form sequences of independent and identically distributed random variables.
Given the mean values of interarrival and service times, we are interested in
evaluating the mean cycle time of the system. In what follows, the mean cycle
time is used in reference to the mean value of the time interval between two
successive departures of customers from the system as the number of customers
tends to infinity. The inverse of the mean cycle time, which implies the mean
number of customers leaving the system per unit time, is referred to as system
throughput.

Among other characteristics including, in particular, the mean waiting time of
customers, both the mean cycle time and the throughput present system
performance measures commonly used in the analysis of queues. Note that the
problem of evaluating the mean waiting time is known as rather difficult; it
allows for the exact solution in an explicit form only for particular queueing
systems under certain restrictions on customer arrival and service processes.
One can find an overview of related results in \cite{Kleinrock1976Queueing} (see, also, 
\cite{Disney1985Queueing} for more recent results and references).

In many cases, the mean cycle time can be calculated exactly under fairly
general assumptions. As an illustration, one can consider results obtained in
\cite{Sacks1960Ergodocity,Loynes1962Thestability} in the context of investigation of stability of queueing 
systems. It has been shown in \cite{Loynes1962Thestability} that for a general single-server 
system with infinite buffer capacity, regardless of whether it is stable or 
not, the mean cycle time can be calculated as the maximum out from the mean 
values of customer interarrival and service times. Moreover, if a tandem 
system of queues with infinite buffers is stable, the intensities of both 
customer arrival and departure processes coincide, and therefore, the mean 
cycle time is equal to the mean interarrival time of customers.

For more complicated queueing systems including tandem queues with finite
buffers and blocking, and fork-join networks there are some techniques which
allow one to derive bounds on the mean cycle time. Specifically, an efficient
approach which relies on results of the theory of large deviations as well as
on the Perron-Frobenius spectral theory has been proposed in 
\cite{Baccelli1991Estimates,Glasserman1995Stochastic}. As another example, one can consider simple bounds in 
\cite{Krivulin1998Bounds,Krivulin1998Monotonicity}, obtained by using an approach essentially based on 
pure algebraic manipulations together with application of bounds on extreme 
values, derived in \cite{Gumbel1954Themaxima,Hartley1954Universal}.

In this paper, we first give quite general conditions for the mean cycle time
in tandem queueing systems with infinite buffers to exist, and show how it can
be calculated through the mean values of the interarrival and service times of
customers. The obtained results are then extended to evaluation of the mean
cycle time in particular tandem systems with finite buffers, which operate
under the manufacturing and communication blocking rules.

As the starting point, we take obvious recursive equations which describe
tandem system dynamics, and then examine their related explicit solution. Our
approach is based on simple algebraic manipulations with the solution,
combined with some classical results providing bounds on the mean value of
maxima of independent and identically distributed random variables. The
approach does not involve taking account of stability of the entire system,
and therefore, offers a unified way of examining both stable and unstable
systems.

The rest of the paper is organized as follows. In Section~\ref{sec-2}, we 
introduce some notations, and consider recursive equations describing the 
dynamics of tandem queueing systems with infinite buffers. Section~\ref{sec-3} 
presents preliminary results including an existence theorem and some 
inequalities. Our main result which provides general existence conditions and 
a simple expression for calculating the mean cycle time is included in 
Section~\ref{sec-4}. The obtained results are then applied to the examination 
of the mean cycle time in particular tandem systems with finite buffers in 
Section~\ref{sec-5}. Finally, Section~\ref{sec-6} offers some concluding 
remarks and discussion.

\section{Tandem Queues with Infinite Buffers}\label{sec-2}

We consider a series of $ M $ single-server queues with infinite buffers
and customers of a single class. Each customer that arrives into the system is
initially placed in the buffer at the $1^{\text{st}}$ server and then has to
pass through all the queues one after the other. Upon the completion of his
service at server $ i $, the customer is instantaneously transferred to
queue $ i+1 $, $ i=1,\dots,M-1 $, and occupies the $(i+1)^{\text{st}}$
server provided that it is free. If the customer finds this server busy, he is
placed in its buffer and has to wait until the service of all his predecessors
is completed.

Denote the time between the arrivals of $n^{\text{th}}$ customer and his
predecessor by $ \tau_{0n} $, and the service time of the $n^{\text{th}}$
customer at server $ i $ by $ \tau_{in} $, $ i=1,\dots,M $,
$ n=1,2,\dots $. Furthermore, let $ D_{0}(n) $ be the $n^{\text{th}}$
arrival epoch to the system, and $ D_{i}(n) $ be the $n^{\text{th}}$
departure epoch from the $i^{\text{th}}$ server. We assume that for each
$ i $, $ i=0,1,\ldots,M $, the sequence
$ \{\tau_{in}|\; n=1,2,\ldots\} $ consists of nonnegative random variables
(r.v.'s).

With the condition that the tandem queueing system starts operating at time
zero, and it is free of customers at the initial time, we put
$ D_{i}(0)=0 $ for all $ i=0,\dots,M $. The recursive equations
representing the system dynamics can readily be written as
\begin{eqnarray*}
D_{0}(n) &=& D_{0}(n-1)+\tau_{0n}, \\
D_{m}(n) &=& \max(D_{m-1}(n),D_{m}(n-1))+\tau_{mn}, \quad m=1,\ldots,M,
\end{eqnarray*}                                   
for all $ n=1,2,\dots $.

The above recursions can be resolved to get
\begin{equation}\label{equ-2}
D_{m}(n)=\max_{1\leq k_{1}\leq\cdots\leq k_{m}\leq n}
\left\{\sum_{j=1}^{k_{1}}\tau_{0j}+\sum_{j=k_{1}}^{k_{2}}\tau_{1j}
+\cdots+\sum_{j=k_{m}}^{n}\tau_{mj}\right\}
\end{equation}
for all $ m=1,\ldots,M $.

We consider the evolution of the system as a sequence of service cycles:
the $1^{\text{st}}$ cycle starts at the initial time, and it is terminated as
soon as the $M^{\text{th}}$ server completes its $1^{\text{st}}$ service, the
$2^{\text{nd}}$ cycle is terminated as soon as this server completes its
$2^{\text{nd}}$ service, and so on. Clearly, the completion time of the
$n^{\text{th}}$ cycle can be represented as $ D_{M}(n) $.

In many applications, one is interested in evaluating the mean service cycle
time of the tandem system, which can also be treated as the mean 
interdeparture time of customers from the system. It is defined as
\begin{equation}\label{equ-3}
\gamma=\lim_{n\to\infty}\frac{1}{n}D_{M}(n)
\end{equation}
provided that the above limit exists. The system throughput $ \pi $
presents another performance measure of interest, which is calculated as the
inverse of the mean cycle time; that is, $ \pi=1/\gamma $.

\section{Preliminary Results}\label{sec-3}

In order to examine the existence of the mean cycle time for the tandem
queueing system, we will apply the next classical theorem which has been
proved in \cite{Kingman1973Subadditive}.
\begin{theorem}\label{the-1}
Let $ \{\zeta_{ln} | \; l,n=0,1,\ldots; \, l<n \} $ be a family of r.v.'s
which satisfy the following properties:

Subadditivity: $ \zeta_{ln}\leq\zeta_{lk}+\zeta_{kn} $ for all
$ l<k<n $;

Stationarity: the joint distributions are the same for both families
$ \{\zeta_{ln} | \; l<n \} $ and $ \{\zeta_{l+1,n+1}| \; l<n \} $; 

Boundedness: for all $ n=1,2,\ldots $, there exists
$ {\mathbb E}[\zeta_{0n}]\geq -cn $ for some positive constant $ c $.

Then there exists a constant $ \gamma $, such that it holds
\begin{enumerate}
\item $ {\displaystyle\lim_{n\to\infty}\zeta_{0n}/n=\gamma} $ with 
probability one (w.p.1),
\item $ {\displaystyle\lim_{n\to\infty}{\mathbb E}[\zeta_{0n}]/n=\gamma} $.
\end{enumerate}
\end{theorem}

Let us now consider some useful inequalities which will be exploited in
evaluation of the mean cycle time in the next section. In what follows, we
assume $ \xi_{1},\ldots,\xi_{n} $ to be independent r.v.'s.

We start with a classical result from \cite{Marcinkiewicz1937Surlesfonctions}, providing an upper bound 
on the mean value of the maximum of cumulative sums
$$
\zeta_{k}=\xi_{1}+\cdots+\xi_{k}
$$
of independent r.v.'s with zero means.
\begin{lemma}\label{lem-2}
If $ {\mathbb E}[\xi_{k}]=0 $, and $ {\mathbb E}|\xi_{k}|^{p}<\infty $
for some $ p>1 $, $ k=1,\ldots,n $, then it holds
$$
{\mathbb E}\left[\max_{1\leq k\leq n}|\zeta_{k}|\right]^{p}
\le
2\left(\frac{p}{p-1}\right)^{p}{\mathbb E}|\zeta_{n}|^{p}.
$$
\end{lemma}

The next inequality has been derived in \cite{Bahr1965Inequalities}. Note that it actually
remains valid under somewhat weaker conditions than that of independence
between the r.v.'s $ \xi_{1},\ldots,\xi_{n} $.
\begin{lemma}\label{lem-3}
If $ {\mathbb E}[\xi_{k}]=0 $, and $ {\mathbb E}|\xi_{k}|^{p}<\infty $
for some $ p $, $ 1\leq p\leq2 $, $ k=1,\ldots,n $, then it holds
$$
{\mathbb E}\left|\zeta_{n}\right|^{p}
\le
\left(2-\frac{1}{n}\right)\sum_{k=1}^{n}{\mathbb E}|\xi_{k}|^{p}.
$$
\end{lemma}

With Lemmas~\ref{lem-2} and \ref{lem-3}, one can prove the following 
statement.
\begin{lemma}\label{lem-4}
If $ {\mathbb E}[\xi_{k}]=0 $, and $ {\mathbb E}[\xi_{k}^{2}]<\infty $,
$ k=1,\ldots,n $, then it holds
$$
{\mathbb E}\left[\max_{1\leq k\leq n}\zeta_{k}\right]
\le
2\sqrt{\frac{2(2n-1)}{n}}
\left(\sum_{k=1}^{n}{\mathbb E}[\xi_{k}^{2}]\right)^{1/2}.
$$
\end{lemma}

\begin{proof}
First note that
$$
{\mathbb E}\left[\max_{1\leq k\leq n}\zeta_{k}\right]
\leq
{\mathbb E}\left[\max_{1\leq k\leq n}|\zeta_{k}|\right]
\leq
\left({\mathbb E}\left[\max_{1\leq k\leq n}|\zeta_{k}|\right]^{2}
\right)^{1/2}.
$$

By applying Lemma~\ref{lem-2} with $ p=2 $, and then Lemma~\ref{lem-3}, we 
get
$$
{\mathbb E}\left[\max_{1\leq k\leq n}|\zeta_{k}|\right]^{2}
\leq
8{\mathbb E}[\zeta_{n}^{2}]
\leq
8\left(2-\frac{1}{n}\right)\sum_{k=1}^{n}{\mathbb E}[\xi_{k}^{2}].
$$
Finally, extracting square root leads us to the desired result.
\end{proof}

Now suppose that $ \xi_{1},\ldots,\xi_{n} $ present independent and
identically distributed (i.i.d.) r.v.'s. With this condition, in particular,
the inequality in Lemma~\ref{lem-4} takes the form
$$
{\mathbb E}\left[\max_{1\leq k\leq n}\zeta_{k}\right]
\le
2\sqrt{2(2n-1){\mathbb E}[\xi_{1}^{2}]}.
$$

The next result obtained in \cite{Gumbel1954Themaxima,Hartley1954Universal} offers an upper bound for the 
expected value of maximum of i.i.d. r.v.'s.
\begin{lemma}\label{lem-5}
If $ {\mathbb E}[\xi_{1}]<\infty $ and $ {\mathbb D}[\xi_{1}]<\infty $, 
then it holds
$$
{\mathbb E}\left[\max_{1\leq k\leq n}\xi_{k}\right]
\le
{\mathbb E}[\xi_{1}]+\frac{n-1}{\sqrt{2n-1}}\sqrt{{\mathbb D}[\xi_{1}]}.
$$
\end{lemma}

Assuming $ \xi_{1},\ldots,\xi_{n} $ to be i.i.d. r.v.'s, let us introduce
the notation
$$
\zeta_{lk}=\xi_{l}+\xi_{l+1}+\cdots+\xi_{k}
$$
with $ 1\leq l\leq k\leq n $, and consider the following statement.
\begin{lemma}\label{lem-6}
If $ {\mathbb E}[\xi_{1}]=a\leq0 $, and $ {\mathbb D}[\xi_{1}]<\infty $, 
then it holds
$$
{\mathbb E}\left[\max_{1\leq l\leq k\leq n}\zeta_{lk}\right]
\le
{\mathbb E}[\xi_{1}]+\left(4\sqrt{2(2n-1)}
+\frac{n-1}{\sqrt{2n-1}}\right)\sqrt{{\mathbb D}[\xi_{1}]}.
$$
\end{lemma}

\begin{proof}
Simple algebraic manipulations give
\begin{eqnarray*}
\max_{1\leq l\leq k\leq n}\zeta_{lk}
&=&
\max_{1\leq l\leq k\leq n}
\left\{\sum_{i=1}^{k}\xi_{i}+\sum_{i=1}^{l-1}(-\xi_{i})\right\} \\
&\leq&
\max_{1\leq l\leq k\leq n}
\left\{\sum_{i=1}^{k}(\xi_{i}-a)+\sum_{i=1}^{l}(-\xi_{i}+a)+\xi_{l}\right\} \\
&\leq&
\max_{1\leq k\leq n}\sum_{i=1}^{k}(\xi_{i}-a)
+
\max_{1\leq k\leq n}\sum_{i=1}^{k}(-\xi_{i}+a)
+
\max_{1\leq k\leq n}\xi_{k}.
\end{eqnarray*}

Proceeding to expectation, with
$ {\mathbb E}(\xi_{1}-a)^{2}={\mathbb D}[\xi_{1}] $, we have from 
Lemmas~\ref{lem-4} and \ref{lem-5}
\begin{eqnarray*}
{\mathbb E}\left[\max_{1\leq l\leq k\leq n}\zeta_{lk}\right]
&\leq&
4\sqrt{2(2n-1){\mathbb D}[\xi_{1}]}
+{\mathbb E}[\xi_{1}]+\frac{n-1}{\sqrt{2n-1}}\sqrt{{\mathbb D}[\xi_{1}]} \\
&=&
{\mathbb E}[\xi_{1}]+\left(4\sqrt{2(2n-1)}
+\frac{n-1}{\sqrt{2n-1}}\right)\sqrt{{\mathbb D}[\xi_{1}]}.
\qedhere
\end{eqnarray*}
\end{proof}

\section{Exact Evaluation of the Mean Cycle Time}\label{sec-4}

We are now in a position to prove our main result which can be formulated as
follows.
\begin{theorem}\label{the-7}
Suppose that $ \{\tau_{in}|\; n=1,2,\ldots\} $, $ i=0,1,\ldots,M, $ are
mutually independent sequences of i.i.d. r.v.'s with
$ 0\leq{\mathbb E}[\tau_{i1}]<\infty $.

Then the limit at (\ref{equ-3}) exists w.p.1, and if
$ {\mathbb D}[\tau_{i1}]<\infty $, it is given by
\begin{equation}\label{equ-4}
\gamma=\max_{0\leq i\leq M}{\mathbb E}[\tau_{i1}].
\end{equation}
\end{theorem}

\begin{proof}
First, we have to verify the existence of the limit at (\ref{equ-3}). In order 
to apply Theorem~\ref{the-1}, let us denote
\begin{equation}\label{equ-5}
\zeta_{ln}=
\max_{l<k_{1}\leq\cdots\leq k_{M}\leq n}
\left\{\sum_{j=l+1}^{k_{1}}\tau_{0j}+\sum_{j=k_{1}}^{k_{2}}\tau_{2j}
+\cdots+\sum_{j=k_{M}}^{n}\tau_{Mj}\right\}
\end{equation}
for each $ l,n $, $ 0\leq l<n $, and note that we can now write
$$
D_{M}(n)=\zeta_{0n}.
$$

With simple algebraic manipulations, it is not difficult to verify that the
family $ \{\zeta_{ln}|\; l<n\} $ defined by (\ref{equ-5}) is subadditive.
Since $ \tau_{i1},\tau_{i2},\ldots $, are i.i.d. r.v.'s for each
$ i=0,1,\ldots,M $, the family also possesses the stationarity property.
Finally, boundedness follows from the condition
$ 0\leq{\mathbb E}[\tau_{i1}]<\infty $ which immediately results in
$ {\mathbb E}[\zeta_{0n}]={\mathbb E}[D_{M}(n)]\geq0 $.

Therefore, one can apply Theorem~\ref{the-1} so as to conclude that the limit 
at (\ref{equ-3}) exists w.p.1, and it can be calculated as
$$
\gamma=\lim_{n\to\infty}\frac{1}{n}{\mathbb E}[D_{M}(n)].
$$

Suppose that the maximum at (\ref{equ-4}) is achieved at some $ i=m $.
Consider the completion time $ D_{M}(n) $ and represent it in the form
$$
D_{M}(n)
= \max_{1\leq k_{1}\leq\cdots\leq k_{M}\leq n}
\left\{\sum_{j=1}^{k_{1}}\tau_{0j}+\sum_{j=k_{1}}^{k_{2}}\tau_{1j}
+\cdots+\sum_{j=k_{M}}^{n}\tau_{Mj}\right\}
= \sum_{j=1}^{n}\tau_{mj}+\mu,
$$
where
\begin{eqnarray}
\mu
&=& 
\max_{1\leq k_{1}\leq\cdots\leq k_{M}\leq n}
\Biggl\{\sum_{j=1}^{k_{1}}(\tau_{0j}-\tau_{mj}) \nonumber \\
& &\mbox{}+\sum_{j=k_{1}}^{k_{2}}(\tau_{1j}-\tau_{mj})
+\cdots+\sum_{j=k_{M}}^{n}(\tau_{Mj}-\tau_{mj}) \nonumber \\
& &\mbox{}+\tau_{mk_{1}}+\cdots+\tau_{mk_{M}}\Biggr\}.\label{equ-6}
\end{eqnarray}

Now we can write
$$
\frac{1}{n}{\mathbb E}[D_{M}(n)]
={\mathbb E}[\tau_{m1}]+\frac{1}{n}{\mathbb E}[\mu].
$$

Let us examine the expected value $ {\mathbb E}[\mu] $. With
$ k_{1}=\cdots=k_{m}=1 $, and $ k_{m+1}=\cdots=k_{M}=n $, we have from
(\ref{equ-6})
$$
\mu\geq
\tau_{01}+\tau_{11}+\cdots+\tau_{m-1,1}+\tau_{m+1,n}+\cdots+\tau_{Mn}\geq0,
$$
and so $ {\mathbb E}[\mu]\geq0 $.

On the other hand, simple algebra gives us an obvious upper bound for
$ \mu $ in the form
\begin{eqnarray*}
\mu
&\leq&
\max_{1\leq k_{1}\leq n}\sum_{j=1}^{k_{1}}(\tau_{0j}-\tau_{mj}) \\
& &\mbox{}+\max_{1\leq k_{1}\leq k_{2}\leq n}
\sum_{j=k_{1}}^{k_{2}}(\tau_{1j}-\tau_{mj})
+\cdots+
\max_{1\leq k_{M}\leq n} 
\sum_{j=k_{M}}^{n}(\tau_{Mj}-\tau_{mj}) \\
& &\mbox{}+
M\max\left(\tau_{m1},\ldots,\tau_{mn}\right).
\end{eqnarray*}

With the condition that $ {\mathbb E}(\tau_{i1}-\tau_{m1})\leq0 $ for
all $ i=0,1,\ldots,M $, one can apply Lemma~\ref{lem-6} to the first 
$ M+1 $ terms on the right-hand side, and then Lemma~\ref{lem-5} to the 
last one so as to get
\begin{eqnarray*}
{\mathbb E}[\mu]
&\leq&
\sum_{i=0}^{M}\left({\mathbb E}(\tau_{i1}-\tau_{m1})
+\left(4\sqrt{2(2n-1)}+\frac{n-1}{\sqrt{2n-1}}\right)
\sqrt{{\mathbb D}(\tau_{i1}-\tau_{m1})}\right) \\
& &\mbox{}+M\left({\mathbb E}[\tau_{m1}]
+\frac{n-1}{\sqrt{2n-1}}\sqrt{{\mathbb D}[\tau_{m1}]}\right) \\
&=&
\sum_{\substack{i=0\\i\ne m}}^M{\mathbb E}[\tau_{i1}]
+
\left(4\sqrt{2(2n-1)}+\frac{n-1}{\sqrt{2n-1}}\right)
\sum_{\substack{i=0\\i\ne m}}^M\sqrt{{\mathbb D}(\tau_{i1}-\tau_{m1})} \\
& &\mbox{}+M\frac{n-1}{\sqrt{2n-1}}\sqrt{{\mathbb D}[\tau_{m1}]},
\end{eqnarray*}
and therefore,
$$
{\mathbb E}[\mu]
\leq\sum_{\substack{i=0\\i\ne m}}^M{\mathbb E}[\tau_{i1}]+O(\sqrt{n}).
$$

Finally, we have the double inequality
$$
{\mathbb E}[\tau_{m1}]\leq\frac{1}{n}{\mathbb E}[D_{M}(n)]
\leq
{\mathbb E}[\tau_{m1}]
+\frac{1}{n}\sum_{\substack{i=0\\i\ne m}}^M{\mathbb E}[\tau_{i1}]
+\frac{O(\sqrt{n})}{n},
$$
and with $ n\rightarrow\infty $, immediately arrive at (\ref{equ-6}).
\end{proof}

\begin{corollary}\label{cor-8}
Under the same conditions as in Theorem~\ref{the-7}, if at least one of the 
expectations $ {\mathbb E}[\tau_{i1}] $, $ i=0,1,\ldots,M $, is positive, 
then it holds
$$
\pi = \left(\max_{0\leq i\leq M}{\mathbb E}[\tau_{i1}]\right)^{-1}.
$$
\end{corollary}

\section{Tandem Queues with Finite Buffers}\label{sec-5}
In this section, we show how the above approach can be applied to the analysis 
of tandem systems which include queues with finite buffers. Because of limited 
buffer capacity, servers in the systems may be blocked according to one of the 
blocking rules \cite{Krivulin1995Amax-algebra}. Below we present examples of systems with 
manufacturing blocking and communication blocking which are most commonly
encountered in practice.

Let us consider a system which consists of two queues in tandem. Suppose that
the buffer at the first server is infinite, while that at the second server is
finite. The customers arriving to the system have to pass through the queues
consecutively, and then leave the system.

First we suppose that the system operates under the manufacturing blocking
rule. With this type of blocking, if upon completion of a service, the first
server sees the buffer of the second one is full, it cannot be freed and has
to remain busy until the second server completes its current service to
provide a free space in its buffer.

For simplicity, let us assume the finite buffer to have capacity $ 0 $.
With the notations introduced above, one can represent the dynamics of the
system by the equations
\begin{eqnarray}
D_{0}(n) &=& D_{0}(n-1)+\tau_{0n}, \nonumber \\
D_{1}(n) &=& \max(\max(D_{0}(n),D_{1}(n-1))+\tau_{1n},D_{2}(n-1)), 
\label{equ-7} \\
D_{2}(n) &=& \max(D_{1}(n),D_{2}(n-1))+\tau_{2n}, \nonumber
\end{eqnarray} 
for all $ n=1,2,\dots $.

Note that from the second equation, we have $ D_{1}(n)\geq D_{2}(n-1) $, and
therefore, the third equation can be reduced to
$$
D_{2}(n) = D_{1}(n)+\tau_{2n}.
$$

Clearly, under appropriate conditions, both $ {\mathbb E}[D_{1}(n)]/n $ 
and $ {\mathbb E}[D_{2}(n)]/n $ have a common limit $ \gamma $ as 
$ n $ tends to $ \infty $, which can be considered as the mean cycle 
time of the system.

By resolving the recursive equations, we get
$$
D_{1}(n)
=\max_{1\leq k\leq n}\left\{\sum_{j=1}^{k}\tau_{0j}+\tau_{1k}
+\sum_{j=k}^{n-1}\max(\tau_{1,j+1},\tau_{2j})\right\}.
$$

As it is easy to verify, $ D_{1}(n) $ satisfies the double inequality
\begin{equation}\label{equ-8}
L(n)-\max(\tau_{1,n+1},\tau_{2n})
\leq D_{1}(n)
\leq U(n),
\end{equation}
where
\begin{eqnarray*}
L(n)
&=& 
\max_{1\leq k\leq n}\left\{\sum_{j=1}^{k}\tau_{0j}
+\sum_{j=k}^{n}\max(\tau_{1,j+1},\tau_{2j})\right\}, \\
U(n)
&=& 
\max_{1\leq k\leq n}\left\{\sum_{j=1}^{k}\tau_{0j}
+\sum_{j=k}^{n}\max(\tau_{1j},\tau_{2,j-1})\right\}.
\end{eqnarray*}

Taking into account that both $ L(n) $ and $ U(n) $ actually
have the form of (\ref{equ-2}), one can see that, under the same conditions as 
in Theorem~\ref{the-7}, it holds
$$
\lim_{n\to\infty}\frac{1}{n}{\mathbb E}[L(n)]
=
\lim_{n\to\infty}\frac{1}{n}{\mathbb E}[U(n)]
=
\max({\mathbb E}[\tau_{01}],{\mathbb E}\max(\tau_{11},\tau_{21})).
$$

Finally, proceeding to mean value in both sides of (\ref{equ-8}), divided by
$ n $, we conclude that the mean cycle time is given by
$$
\gamma=\max({\mathbb E}[\tau_{01}],{\mathbb E}\max(\tau_{11},\tau_{21})).
$$

Let us now assume the system to follow the communication blocking rule. This
type of blocking requires the first server not to initiate service of a
customer if the buffer of the second server is completed. With the finite
buffer having capacity $ 0 $, the system dynamics is described by the same
recursions as above, except for equation (\ref{equ-7}) which now takes the 
form
$$
D_{1}(n) = \max(D_{0}(n),D_{1}(n-1),D_{2}(n-1))+\tau_{1n}.
$$

Resolving the recursive equations leads us to the expression
$$
D_{2}(n)
=\max_{1\leq k\leq n}\left\{\sum_{j=1}^{k}\tau_{0j}
+\sum_{j=k}^{n}(\tau_{1j}+\tau_{2j})\right\}.
$$

Under the same conditions as in Theorem~\ref{the-7}, we get the mean cycle 
time
$$
\gamma
=\max({\mathbb E}[\tau_{01}],{\mathbb E}[\tau_{11}]+{\mathbb E}[\tau_{21}]).
$$

\section{Concluding Remarks}\label{sec-6}

First note that with similar arguments as used in the proof of 
Theorem~\ref{the-7} and related lemmas, one can verify that the statement of 
the theorem remains valid with the condition 
$ {\mathbb E}[\tau_{i1}^{\alpha}]<\infty $ for some $ \alpha>1 $, 
instead of $ {\mathbb D}[\tau_{i1}]<\infty $.

Furthermore, the proof of the theorem does not actually require that for each
$ n $, r.v.'s $ \tau_{in} $ with $ i=0,1,\ldots,M $, be independent.
This allows one to apply obtained results to tandem queueing systems with
dependence for each customer between his interarrival and service times,
including tandem queues with identical service times at each server.

Theorem~\ref{the-1} actually offers more general existence conditions for the 
mean cycle time, which imply stationarity of the sequence 
$ \{(\tau_{0n},\tau_{1n},\ldots\tau_{Mn})|\; n=1,2,\ldots\} $ in place of
independence conditions in Theorem~\ref{the-7}.

Let us consider recursive equations describing the dynamics of the tandem
queues above, and note that the symbol $ \tau_{in} $ can be thought of as 
the $n^{\text{th}}$ service time at server $ i $ rather than the service 
time of the $n^{\text{th}}$ customer at the server. Since in this case, the 
order in which customers are selected from a queue for service is of no 
concern, the equation also describes the dynamics of systems with any queueing 
disciplines not permitting the preempting of service, and therefore, 
Theorem~\ref{the-7} can trivially be extended to such systems.

Finally note that the proof of the theorem provides us with bounds on
$ {\mathbb E}[D_{M}(n)] $, as well as an upper bound of order 
$ n^{-1/2} $ for the convergence rate of $ {\mathbb E}[D_{M}(n)]/n $ 
to the mean cycle time as $ n $ tends to $ \infty $.

\bibliographystyle{utphys}

\bibliography{On_evaluation_of_the_mean_service_cycle_time_in_tandem_queueing_systems}

\end{document}